\documentclass[12pt]{amsart}
\usepackage{amssymb}
\usepackage{amsmath, amscd}
\usepackage{amsthm}
\usepackage{amstext}
\usepackage[table,xcdraw]{xcolor}
\usepackage[colorlinks=true, linkcolor=blue,urlcolor=blue]{hyperref}

\newtheorem{theorem}{Theorem}[section]

\newtheorem{proposition}[theorem]{Proposition}
\newtheorem{lemma}[theorem]{Lemma}

\newtheorem{corollary}[theorem]{Corollary}
\theoremstyle{definition}

\newtheorem{example}[theorem]{Example}
\newtheorem{definition}[theorem]{Definition}

\begin{document}

\author[A. Moussavi]{Ahmad Moussavi}
\address{Department of Mathematics, Tarbiat Modares University, 14115-111 Tehran Jalal AleAhmad Nasr, Iran}
\email{moussavi.a@modares.ac.ir; moussavi.a@gmail.com}

\author[P. Danchev]{Peter Danchev}
\address{Institute of Mathematics and Informatics, Bulgarian Academy of Sciences, 1113 Sofia, Bulgaria}
\email{danchev@math.bas.bg; pvdanchev@yahoo.com}

\author[A. Javan]{Arash Javan}
\address{Department of Mathematics, Tarbiat Modares University, 14115-111 Tehran Jalal AleAhmad Nasr, Iran}
\email{a.darajavan@modares.ac.ir; a.darajavan@gmail.com}

\author[O. Hasanzadeh]{Omid Hasanzadeh}
\address{Department of Mathematics, Tarbiat Modares University, 14115-111 Tehran Jalal AleAhmad Nasr, Iran}
\email{o.hasanzade@modares.ac.ir; hasanzadeomiid@gmail.com}

\title{On strongly $\Delta$-clean rings}

\maketitle

%\date{.}

%\begin{document}

%\maketitle

\begin{abstract}
This study explores in-depth the structure and properties of the so-called {\it strongly $\Delta$-clean rings}, that is a novel class of rings in which each ring element decomposes into a sum of a commuting idempotent and an element from the subset $\Delta(R)$. Here, $\Delta(R)$ stands for the extension of the Jacobson radical and is defined as the maximal subring of $J(R)$ invariant under the unit multiplication. We present a systematic framework for these rings by detailing their foundational characteristics and algebraic behavior under standard constructions, as well as we explore their key relationships with other well-established ring classes. Our findings demonstrate that all strongly $\Delta$-clean rings are inherently strongly clean and $\Delta U$, but under centrality constraints they refine the category of uniquely clean rings. Additionally, we derive criteria for the strong $\Delta$-clean property in triangular matrix rings, their skew analogs, trivial extensions, and group rings. The analysis reveals deep ties to boolean rings, local rings, and quasi-duo rings by offering new structural insights in their algebraic characterization.
\end{abstract}

\section{Introduction and Background}

Throughout this paper, all rings are assumed to be unital and associative. Almost all symbols and concepts are traditional and are consistent with the well-known books \cite{lamf} and \cite{lame}. As usual, the Jacobson radical, the set of nilpotent elements, the set of idempotent elements, and the set of units of \( R \) are denoted by \( J(R) \), \( Nil(R) \), \( Id(R) \), and \( U(R) \), respectively. Additionally, we write $M_n(R)$ and $T_n(R)$ for the $n \times n$ matrix ring and the $n \times n$ upper triangular matrix ring, respectively. Standardly, a ring is termed {\it abelian} if each idempotent element is central.

Imitating \cite{2} and \cite{7}, an element $r \in R$ is said to be {\it clean} if there is an idempotent $e \in R$ and an unit $u \in R$ such that $r=e+u$. Such an element $r$ is further called {\it strongly clean} if the existing idempotent and unit can be chosen such that $ue=eu$. Accordingly, a ring is called {\it clean} (respectively, {\it strongly clean}) if each of its elements is clean (respectively, strongly clean). For a comprehensive investigation of this class of rings are written too many papers as, for a more detailed information, we refer to these cited in the current bibliography.

On the other hand, mimicking \cite{diesl}, an element $r\in R$ is said to be {\it nil-clean} if there is an idempotent $e \in R$ and a nilpotent $b \in R$ such that $r=e+b$. Such an element $r$ is further called {\it strongly nil-clean} if the existing idempotent and nilpotent can be chosen such that $be=eb$. A ring is called {\it nil-clean} (respectively, {\it strongly nil-clean}) if each of its elements is nil-clean (respectively, strongly nil-clean).

On the other side, following \cite{csj, chensjc}, \textit{strongly J-clean} rings are those rings in which each element can be written as the sum of an idempotent and an element from the Jacobson radical $J(R)$ that commute.

By considering and analyzing the aforementioned definitions along with the fact that $\Delta(R)$ is a (possibly proper) subset of $J(R)$, which is not necessarily an ideal and does not retain the desirable properties such as those of $Nil(R)$, a natural question arises regarding the structure of those rings $R$ in which every element can be expressed as the sum of an idempotent element and an element from $\Delta(R)$ that commute with each other. The primary aim of the present work is to examine such rings and to undertake a thorough investigation of their structural properties.

The examination of clean-type rings has attracted considerable attention due to their deep structural properties and their interrelations with various notions in non-commutative ring theory. Among these, the concepts of strongly clean and strongly nil-clean rings has led to numerous developments and refinements in understanding how elements of a ring can be decomposed into simpler algebraic components. Building on this rich framework, the present paper naturally introduces and explores the notion of \textit{strongly $\Delta$-clean} rings in which each element is expressed as a sum of an idempotent and an element from $\Delta(R)$ that commute with one another as the terminology logically follows that from the exposition presented above. Certainly, all strongly $J$-clean rings are themselves strongly $\Delta$-clean, but the reverse implication is, expectably, {\it not} true in all generality (see Problem 1).

In this vein, the set $\Delta(R)$ defined as
\[
\Delta(R) = \{x \in R \mid 1 - xu \in U(R) \text{ for all } u \in U(R)\},
\]
has been intensively studied as a generalization of the Jacobson radical (see \cite[Exercise 4.24]{lame} and \cite{lmr}). Although $\Delta(R)$ always properly contains $J(R)$, it is {\it not} necessarily an ideal and thus some of the favorable properties of the classical nil-radicals lack. Nonetheless, it plays an important role as being the largest subring of $J(R)$ that is closed under multiplication by units.

The goal of this article is to initiate a systematic study of strongly $\Delta$-clean rings by developing their basic properties and structure theory. In particular, we compare them with well-known classes such as strongly clean rings, uniquely clean rings, and $\Delta U$ rings. We establish several characterizations, investigate their behavior under standard ring constructions, and identify a variety of structural restrictions that govern such rings. These include conditions under which matrix rings, trivial extensions, group rings, and skew polynomial rings exhibit the strongly $\Delta$-clean property. The results provided herein not only generalize existing knowledge in the subject, but also open the door for further explorations into the role of $\Delta(R)$ in ring theory.

We are now planning to give a brief program of our main material established in the sequel: In the next Section 2, we succeed in establishing several fundamental properties and characterizations of strongly $\Delta$-clean rings from multiple perspectives (see, for example, Lemma \ref{lemma 0}, \ref{lemma 2}, \ref{lemma 3}, \ref{lemma 4}, \ref{dedkind finite}, Corollary \ref{corner ring}, and Proposition \ref{2.2}). In the subsequent Section 3, in the course of our investigation, we succeed in establishing several principal and distinctive characterizing properties of strongly $\Delta$-clean rings. These properties are primarily formulated and rigorously proven in Theorems \ref{1}, \ref{cor 2}, \ref{cor 3}, \ref{boolean}, \ref{triangular}, and \ref{11} accompanied by a number of auxiliary results and related statements that further illuminate the structural and algebraic behavior of such rings. In Section 4, we give some extensions of strongly $\Delta$-clean rings; for instance, skew polynomial extensions, skew triangular matrix extensions, trivial extensions and Morita contexts. In final Section 5, we investigate the conditions under which a group ring becomes strongly $\Delta$-clean subject to certain restrictions on the underlying group and ring (see, for instance, Theorem \ref{1.6} and Proposition \ref{group}). We finish our work with two challenging problems of some interest and importance which, hopefully, will stimulate a further study on the topic (see Problems 1 and 2).

\section{Examples and Basic Properties}

We start here with our pivotal instrument.

\begin{definition}
We say that $R$ is a {\it strongly $\Delta$-clean} ring if every element of $R$ is the sum of an idempotent from $R$ and an element from $\Delta(R)$ that commute with each other. Such a sum's presentation is also said to be a {\it strongly $\Delta$-clean} representation (or just {\it S$\Delta$C} representation for short).
\end{definition}

A few more technicalities are now on hold.

\begin{lemma}\label{1.1}
For any ring \(R\), the following equality is true:
	\[
	U(R) + \Delta(R) = U(R).
	\]
\end{lemma}

\begin{proof}
We know that \(U(R) \subseteq U(R) + J(R) \subseteq U(R) + \Delta(R)\). To treat the reciprocal inclusion, given \(x \in U(R) + \Delta(R)\), we may write \(x = u+r\), where \(u \in U(R)\) and \(r \in \Delta(R)\). One easily checks that \(x = u + r = u(1+u^{-1}r) \in U(R) \), as required.
\end{proof}

The following claim can easily be proven, so we omit the details leaving them to the interested reader for a direct check.

\begin{lemma}\label{lemma 0}
(1) Suppose \( R = \prod_{i \in I} R_i \). Then, \( R \) is a strongly $\Delta$-clean ring if, and only if, for each \( i \in I \), \( R_i \) is a strongly $\Delta$-clean ring.

(2) Suppose \( R \) is a strongly $\Delta$-clean ring and \( I \) is an ideal of \( R \) such that \( I \subseteq J(R) \). Then, \( R/I \) is a strongly $\Delta$-clean ring (in particular, $R/J(R)$ is a strongly $\Delta$-clean ring).
\end{lemma}

\begin{lemma}\label{1.5}
Every strongly $\Delta$-clean ring is strongly clean.	
\end{lemma}

\begin{proof}
Suppose that $r\in R$ is a strongly $\Delta$-clean element. Then, there exist an idempotent $e\in R$ and an element $d\in \Delta(R)$ such that $r=e+d$ with $ed=de$. Hence, $r=(1-e)+(2e-1+d)$. As $(2e-1)^2=1$, we plainly see that $2e-1+d\in U(R)$. Thus, $r$ is strongly clean, as needed.
\end{proof}

\begin{example}
(1) Every element in $\Delta(R)$ is strongly $\Delta$-clean.

(2) $u\in U(R)$ and $u$ is strongly $\Delta$-clean if, and only if, $u-1\in \Delta(R)$.

(3) $a\in R$ is strongly $\Delta$-clean if, and only if, $1-a\in R$ is strongly $\Delta$-clean.	
\end{example}

\begin{proof}
It is straightforward, so we drop off the arguments.
\end{proof}

\begin{lemma}\label{lemma 1}
For every \( e \in Id(R) \) and \( d \in \Delta(R) \), the containment \( 2ed \in \Delta(R) \) holds.
\end{lemma}

\begin{proof}
For every \( e = e^2 \in R \), we have $(2e-1)^2=1$ and hence \(2e-1\in U(R) \). So, applying \cite[Lemma 1(2)]{lmr}, it follows that, for each \( d \in \Delta(R) \), the inclusion \((2e-1)d\in \Delta(R) \) is valid. Since \( \Delta(R) \) is a subring of \( R \), we obtain \( 2ed \in \Delta(R) \), as required.
\end{proof}

\begin{lemma} \label{lemma 2}
Let \( R \) be a strongly $\Delta$-clean ring, and \( a \in R \). If \( a^2 \in \Delta(R) \), then \( a \in \Delta(R) \).
\end{lemma}

\begin{proof}
Assume that \( a = e + d \) is a strongly $\Delta$-clean representation. We write \( a^2 = e + 2ed + d^2 \). Thanks to Lemma \ref{lemma 1}, we get \[ e = a^2 - 2ed - d^2 \in \Delta(R) \cap Id(R) = 0 ,\] which implies \( e = 0 \). Thus, \( a = d \in \Delta(R) \), as wanted.
\end{proof}

As two consequences, we yield:

\begin{corollary} \label{subset nil}
Let \( R \) be a strongly $\Delta$-clean ring. Then, \( \text{Nil}(R) \subseteq \Delta(R) \).
\end{corollary}

\begin{corollary} \label{com idim}
Let $R$ be a strongly $\Delta$-clean ring. Then, for every $r \in R$ and $e \in \mathrm{Id}(R)$, $er - re \in \Delta(R)$.
\end{corollary}

\begin{proof}
Since $[er(1-e)]^2 = 0 = [(1-e)re]^2$, one sees that Corollary \ref{subset nil} employs to write $er(1-e), (1-e)re \in \Delta(R)$. Moreover, since $\Delta(R)$ is closed under addition, one verifies that
\[
er - re = er(1-e) - (1-e)re \in \Delta(R), \qedhere
\]
as desired.
\end{proof}

\begin{lemma} \label{lemma 3}
Let \( R \) be a strongly $\Delta$-clean ring. Then, for any \( a \in R \), the containment \( a - a^2 \in \Delta(R) \) is fulfilled.
\end{lemma}

\begin{proof}
Assume \( a = e + d \) is a strongly $\Delta$-clean representation. We, thus, arrive at

\[ a - a^2 = (d - d^2) - 2ed. \]

But, according to Lemma \ref{lemma 1}, we can conclude that \( a - a^2 \in \Delta(R) \), as pursued.
\end{proof}

Let $R$ be a ring and suppose $\alpha : R \to R$ is a ring endomorphism; then, $R[x; \alpha]$ denotes the {\it ring of skew polynomials} over $R$ with multiplication defined by $xr = \alpha(r)x$ for all $r \in R$. In particular, $R[x] = R[x; 1_R]$ is the {\it ring of polynomials} over $R$.

\begin{example} \label{example 1}
Let \( R \) be an arbitrary ring. Then, \( R[x] \) is {\it not} a strongly $\Delta$-clean ring.
\end{example}

\begin{proof}
Owing to Lemma \ref{lemma 3}, we discover that \( x - x^2 \in \Delta(R[x]) \), and so \( 1 - x + x^2 \in U(R[x]) \), which is a contradiction, as asked for.
\end{proof}

Let $Nil_{*}(R)$ denote the {\it prime radical} (or, in other terms, the {\it lower nil-radical}) of a ring $R$, i.e., the intersection of all prime ideals of $R$. We know that $Nil_{*}(R)$ is a nil-ideal of $R$. Recall that a ring $R$ is said to be {\it 2-primal} if $Nil_{*}(R) = Nil(R)$. For example, each commutative or reduced ring is 2-primal.

\medskip

For an arbitrary endomorphism \( \alpha \) of \( R \), the ring \( R \) is called {\it \(\alpha\)-compatible} if, for any \( a, b \in R \), \( ab = 0 \) holds if, and only if, \( a\alpha(b) = 0 \) (see, e.g., \cite{ann}). In this case, it is apparent that \( \alpha \) is injective.

\medskip

The following equality appears to be worthy of documentation.

\begin{proposition} \label{pro 1}
Let \( R \) be simultaneously a 2-primal and \(\alpha\)-compatible ring. Then, \[\Delta(R[x, \alpha]) = \Delta(R) + Nil_{*}(R[x, \alpha])x.\]
\end{proposition}

\begin{proof}
Assuming \( f = \sum_{i=0}^{n} a_i x^i \in \Delta(R[x, \alpha]) \), then, for every \( u \in U(R) \), we find that \( 1 - uf \in U(R[x, \alpha]) \). Thus, invoking \cite[Corollary 2.14]{cheni}, \( 1 - ua_0 \in U(R) \) and, for each \( 1 \le  i \le  n \), \( ua_i \in Nil_{*}(R) \). Since \( Nil_{*}(R) \) is an ideal, we obtain \( a_0 \in \Delta(R) \) and, for each \( 1 \le  i \le  n \), \( a_i \in Nil_{*}(R) \). Moreover, as \( R \) is a 2-primal ring, \cite[Lemma 2.2]{cheni} applies to get that \( Nil_{*}(R)[x,\alpha] = Nil_{*}(R[x, \alpha]) \).

Conversely, assume \( f \in \Delta(R) + Nil_{*}(R[x, \alpha])x \) and \( u \in U(R[x,\alpha]) \). Then, appealing to \cite[Corollary 2.14]{cheni}, we have \( u \in U(R) + Nil_{*}(R[x, \alpha])x \). Since \( R \) is a 2-primal ring, it must be that \[ 1 - uf \in U(R) + Nil_{*}(R[x, \alpha])x \subseteq U(R[x, \alpha]) ,\] and thus \( f \in \Delta(R[x, \alpha]) \), as required.
\end{proof}

\begin{proposition}\label{pro 2}
Let \( R \) be both a 2-primal and \(\alpha\)-compatible ring, and let \(e = \sum_{i=0}^n e_i x^i \in Id(R[x, \alpha]) \). Then, \( e_0^2 = e_0 \) and, for every \( 1 \le i \le n \), \( e_i \in Nil(R) \) holds.
\end{proposition}

\begin{proof}
It is evident that \( e_0^2 = e_0 \), so it suffices to show that, for each \( 1 \le i \le n \), \( e_i \in Nil(R) \) is true. To that end, since \( e^2 = e \), we have \( e_n \alpha^n(e_n) = 0 \). Likewise, because \( R \) is \(\alpha\)-compatible, in virtue of \cite[Lemma 2.1]{hm}, we find that \( e_n^2 = 0 \).

Now, assume \( f := e - e_n x^n \). Since \( e^2 = e \) and \( e_n \in Nil_*(R) \), we write \( f - f^2 \in Nil_*(R)[x, \alpha] \), and so \(\bar{f} = \bar{f}^2 \in (R/Nil_*(R))[x, \alpha] \). Thus, \( e_{n-1} \alpha^{n-1}(e_{n-1}) \in Nil_*(R) \).

Furthermore, since \( R \) is a \(\alpha\)-compatible ring, we apply \cite[Lemma 2.1]{hm} to deduce \( e_{n-1} \in Nil(R) \). Continuing the process in this manner, it can be shown that, for every \( 1 \le i \le n \), \( e_i \in Nil(R) \) holds, as stated.
\end{proof}

One observes that, in view of Example \ref{example 1}, the ring $R[x]$ is {\it not} strongly $\Delta$-clean. In the next lemma, we show what set can form the elements with strongly $\Delta$-clean representation in $R[x]$. For this purpose, we need a new notation, namely we denote all strongly $\Delta$-clean elements of the ring $R$ by $S\Delta C(R)$.

\begin{lemma}
Let \( R \) be a 2-primal and \(\alpha\)-compatible ring. Then, \[ S\Delta C(R[x, \alpha]) \subseteq S\Delta C(R) + Nil_*(R)[x, \alpha]x .\]
\end{lemma}

\begin{proof}
Write \( f = \sum_{i=0}^n f_i x^i \in S\Delta C(R[x, \alpha]) \) and assume \( f = \sum_{i=0}^n e_i x^i + \sum_{i=0}^n d_i x^i \) is a strongly $\Delta$-clean representation. Then, we derive $f_0=e_0+d_0$ with $e_0d_0=d_0e_0$. Thanking to Propositions \ref{pro 1} and \ref{pro 2}, we infer that \( e_0 = e_0^2 \) and \( d_0 \in \Delta(R) \); thus, clearly, \( f_0 \in S\Delta C(R) \). Moreover, employing Proposition \ref{pro 2}, for every \( 1 \le i \le n \), \( e_i\in Nil_*(R) \). Also, Proposition \ref{pro 1} works to get that \( d_i\in Nil_*(R) \). Hence, we conclude \( f_i = e_i + d_i \in Nil_*(R) \), as needed.
\end{proof}

As usual, a ring $R$ is called {\it reduced} if it contains no non-zero nilpotent elements, that is, $Nil(R)=(0)$.

\begin{lemma} \label{lemma 4}
Let \( R \) be a strongly $\Delta$-clean ring. Then, the factor-ring \( R/J(R) \) is reduced.
\end{lemma}

\begin{proof}
Assume \( x^2 \in J(R) \subseteq \Delta(R) \). Thus, thanks to Lemma \ref{lemma 2}, we discover that \( x \in \Delta(R) \). Let \( r \in R \). Since \( 1 - r^2x^2 \in U(R) \), we know \( u := 1 - rx^2r \in U(R) \). Therefore,

\[
(1 - rx)(1 + xr) = 1 - rx + xr - rx^2r = xr - rx + u.
\]

\medskip

It now suffices to prove that \( xr - rx \in \Delta(R) \). To this target, assume \( r = e + d \) is a strongly $\Delta$-clean representation. Then, it must be that

\[
xr - rx = x(e + d) - (e + d)x = xe - ex + (xd - dx),
\]

where, by virtue of Corollary \ref{com idim}, we detect that $xe - ex \in \Delta(R)$.

On the other hand, since $x, d \in \Delta(R)$ and $\Delta(R)$ is closed under addition, we obtain $xr - rx= (xe - ex) + xd - dx \in \Delta(R)$, as claimed.

Hence,

\[
(1 - rx)(1 + xr) \in U(R) + \Delta(R) \subseteq U(R).
\]

Since \( r \) was arbitrary, we conclude \( x \in J(R) \), as expected.
\end{proof}

\begin{lemma}\label{lemma 5}
Let \( R \) be a strongly $\Delta$-clean ring. Then, \( 2 \in J(R) \).
\end{lemma}

\begin{proof}
Referring to Lemma \ref{lemma 3}, we obtain \( 2 \in \Delta(R) \), which gives \( 4 = 2 + 2 \in \Delta(R) \). Choose \( r \in R \), and let \( r = e + d \) be a strongly $\Delta$-clean representation. Then, utilizing Lemma \ref{lemma 1}, we deduce

\[
1 - 4r = 1 - 4e - 4d = 1 - 2(2e) - 4d \in 1 + \Delta(R) \subseteq U(R).
\]

\medskip

Thus, \( 4 \in J(R) \).

Furthermore, since \( 2^2 =4 \in J(R) \), Lemma \ref{lemma 4} allows us to conclude that \( 2 \in J(R) \), as promised.
\end{proof}

A ring is called \(\Delta U\), provided \(1 + \Delta(R) = U(R)\) (see \cite{kkqt}).

\begin{lemma}\label{du ring}
Every strongly $\Delta$-clean ring is a \(\Delta U\) ring.
\end{lemma}

\begin{proof}
Choose \(u \in U(R)\) and suppose \(u = e + d\) is a strongly $\Delta$-clean representation. Then, one inspects that \(e = u - d \in U(R) + \Delta(R) \subseteq U(R) \cap Id(R) = \{1\}\), as required.
\end{proof}

Let $R$ be a ring, and given $a \in R$. In what follows below, set $ann_{l}a := \{ r \in R : ra = 0\}$ and $ann_{r}a := \{r \in  R: ar = 0\}$.

\begin{lemma}\label{ann}
Let $R$ be a ring and suppose $a = e + d$ is a strongly $\Delta$-clean representation in $R$. Then, $\text{ann}_l(a) \subseteq \text{ann}_l(e)$ and $\text{ann}_r(a) \subseteq \text{ann}_r(e)$.
\end{lemma}

\begin{proof}
Assume $ra = 0$. It follows from Lemma \ref{lemma 1} that there is $d' \in \Delta(R)$ such that $a^2 = e + d'$. Since $ra = 0$, we have $re + rd' = 0$. Now, multiplying by $e$ from the right, we get $re + red' = 0$, so $re(1 + d') = 0$. Since $d' \in \Delta(R)$, it must be that $1 + d' \in U(R)$, which forces $re = 0$. Thus, $r \in \text{ann}_l(e)$. Similarly, it can be shown that $\text{ann}_r(a) \subseteq \text{ann}_r(e)$, as formulated.
\end{proof}

\begin{lemma}
Let $R$ be a ring and $e \in R$ an idempotent. If $a \in eRe$ is a strongly $\Delta$-clean element in $R$, then $a$ is a strongly $\Delta$-clean element in $eRe$.
\end{lemma}

\begin{proof}
Suppose $a = f + d$, where $f = f^2$, $d \in \Delta(R)$ and $fd = df$. Since $1 - e \in \text{ann}_l(a) \cap \text{ann}_r(a)$, with Lemma \ref{ann} at hand we find that $1 - e \in \text{ann}_l(f) \cap \text{ann}_r(f)$, which ensures $(1 - e)f = f(1 - e) = 0$. Thus, $f = ef = fe$. Also, since $a \in eRe$, we have $a = ea = ae = eae$. Now, multiplying $a = f + d$ by $e$ from the left and right sides, we deduce $a = efe + ede$. Note that, since $f = ef = fe$ and $f$ is an idempotent, $efe$ is also an idempotent. That is why, it suffices to show that $ede \in \Delta(eRe)$.

On the other hand, since $f = ef = fe = efe$ and $a = ea = ae = eae$, it is obvious that $$d = ed = de = ede \in \Delta(R) \cap eRe.$$ Now, we prove that $eRe \cap \Delta(R) \subseteq \Delta(eRe)$ always holds. To sustain this, assume $r \in eRe \cap \Delta(R)$ and $u \in U(eRe)$. Then, $$(u + (1 - e))(u^{-1} + (1 - e)) = 1,$$ so $u + (1 - e) \in U(R)$. Since $r \in \Delta(R)$, there is $v \in R$ such that $(1 - (u + (1 - e))r)v = 1$. And since $r \in eRe$, we receive $(1 - ur)v = 1$. Furthermore, multiplying by $e$ from the left and right sides, we derive $(e - ur)eve = e$, which insures $r \in \Delta(eRe)$. Thus, $d \in \Delta(eRe)$ whence $ede \in \Delta(eRe)$, as claimed.
\end{proof}

As an immediate consequence, we yield:

\begin{corollary}\label{corner ring}
Let $R$ be a ring, and let $e \in R$ be an idempotent. If $R$ is a strongly $\Delta$-clean ring, then so is $eRe$.
\end{corollary}

The next assertion is {\it not} too surprising.

\begin{proposition}\label{2.2}
For any ring $R$ and every $n \geq 2$, the matrix ring $M_n(R)$ is not strongly $\Delta$-clean.
\end{proposition}

\begin{proof}
With the aid of Lemma \ref{du ring}, if $M_n(R)$ were strongly $\Delta$-clean, then it would also be $\Delta U$. However, an appeal to \cite[Theorem 2.5]{kkqt} guarantees that, for each $n \geq 2$, the matrix ring $M_n(R)$ is not $\Delta U$. This means that, for each $n \geq 2$, the ring $M_n(R)$ is not strongly $\Delta$-clean, as asserted.
\end{proof}

The following technical claim is useful. Recall that $R$ is an {\it uniquely clean ring}, provided that each element in $R$ has an unique representation as the sum of an idempotent and an unit (see \cite{nzu}).

\begin{lemma}
Let $R$ be a ring. Then, the following four conditions hold:
	
(1) A division ring $R$ is strongly $\Delta$-clean if, and only if, $R \cong \mathbb{F}_2$.
	
(2) A ring $R$ is local and strongly $\Delta$-clean if, and only if, $R/J(R) \cong \mathbb{F}_2$.
	
(3) A semi-simple ring $R$ is strongly $\Delta$-clean if, and only if, $R \cong \mathbb{F}_2 \times \cdots \times \mathbb{F}_2$.
	
(4) If $R$ is semi-local and strongly $\Delta$-clean, then $R/J(R) \cong \mathbb{F}_2 \times \cdots \times \mathbb{F}_2$.
\end{lemma}

\begin{proof}
(1) If $R$ is a division ring, then we know that $\Delta(R)=0$ and also that $Id(R)=\{0,1\}$. Therefore, $U(R)={1}$. The converse is obvious.
	
(2) If $R$ is local and strongly $\Delta$-clean ring, then (1) works. Reciprocally, suppose that $R/J(R) \cong \mathbb{F}_2$. Then, by \cite[Theorem 15]{nzu}, $R$ is an uniquely clean ring and, therefore, strongly $\Delta$-clean (compare with Theorem~\ref{cor 2} listed below).
	
(3) If $R$ is a semi-simple ring, in accordance with the well-known Wedderburnâ-Artin theorem we write $R\cong \prod M_{n_i}(D_i)$, where all $D_i$ are division rings. Consequently, based on Lemma \ref{lemma 0}(1), Corollary \ref{2.2}, and (1), we conclude that $R \cong \mathbb{F}_2 \times \cdots \times \mathbb{F}_2$, as expected. The opposite implication is obvious.
	
(4) It is clear by (3).
\end{proof}

A set $\{e_{ij} : 1 \le i, j \le n\}$ of nonzero elements of $R$ is said to be a system of $n^2$ matrix units if $e_{ij}e_{st} = \delta_{js}e_{it}$, where $\delta_{jj} = 1$ and $\delta_{js} = 0$ for $j \neq s$. In this case, $e := \sum_{i=1}^{n} e_{ii}$ is an idempotent of $R$ and $eRe \cong M_n(S)$, where $$S = \{r \in eRe : re_{ij} = e_{ij}r,~~\textrm{for all}~~ i, j = 1, 2, . . . , n\}.$$
Recall that a ring $R$ is said to be {\it Dedekind-finite} if $ab=1$ is equivalent to $ba=1$ for any $a,b\in R$. In other words, all one-sided inverses in the ring are two-sided.

\medskip

The next affirmation is {\it not} too curious.

\begin{lemma} \label{dedkind finite}
Every strongly $\Delta$-clean ring is Dedekind-finite.
\end{lemma}

\begin{proof}
If we assume the contrary that $R$ is not a Dedekind-finite ring, then there are elements $a, b \in R$ such that $ab = 1$ but $ba \neq 1$. Assuming $e_{ij} = a^i(1-ba)b^j$ and $e =\sum_{i=1}^{n}e_{ii}$, there exists a nonzero ring $S$ such that $eRe \cong M_n(S)$. However, Corollary \ref{corner ring} is a guarantor that $eRe$ is a strongly $\Delta$-clean ring, so that $M_n(S)$ must also be a strongly $\Delta$-clean, which contradicts Corollary \ref{2.2}.
\end{proof}

\section{Main Results}

We begin here with the following first chief criterion.

\begin{theorem}\label{1}
A ring $R$ is strongly $\Delta$-clean if, and only if, $R$ is strongly clean and $\Delta(R) = \{ x \in R \mid 1 - x \in U(R) \}$.
\end{theorem}

\begin{proof}
Suppose that $R$ is strongly $\Delta$-clean. Then, $R$ is strongly clean by Lemma \ref {1.5}. Obviously, $\Delta(R)\subseteq \{ x \in R \mid 1 - x \in U(R) \}$. Now, if $1-x\in U(R)$, we can choose an idempotent $e\in R$ and an element $d\in\Delta(R)$ such that $x=e+d$ with $ed=de$. Thus, $1-e=(1-x)+d\in U(R)$, and so $e=0$. Hence, $x\in \Delta(R)$.

Conversely, assume that $R$ is strongly clean and $\Delta(R) = \{ x \in R \mid 1 - x \in U(R) \}$. For any $a\in R$, we write $-a=e+u$, where $e\in Id(R)$ and $u\in U(R)$ with $eu=ue$. This leads to $a=(1-e)+(1+u)$. Therefore, $a\in R$ is strongly $\Delta$-clean, as desired.
\end{proof}

Note that, in view of \cite[Theorem 2.4]{cyz}, the ring $M_{2}(\hat{\mathbb{Z}}_{p})$ is strongly clean that is {\it not} strongly $\Delta$-clean in virtue of Corollary \ref {2.2}.

\medskip

We are now ready to proceed by proving our second major statement.

\begin{theorem}\label{cor 2}
Let $R$ be a ring. Then, the following are equivalent:

(1) $R$ is an uniquely clean ring.

(2) $R$ is a strongly $\Delta$-clean with all idempotents central.
\end{theorem}

\begin{proof}
$(1)\Rightarrow(2)$. Assume \(R\) is an uniquely clean ring. Looking at \cite[Lemma 4]{nzu}, every idempotent in \(R\) is central. Also, with the help of \cite[Theorem 20]{nzu}, for every \(a \in R\), there is an unique idempotent \(e\) such that \(a - e \in J(R) \subseteq \Delta(R)\). Thus, there is \(d \in \Delta(R)\) such that \(a = e + d\). Since all idempotents are central, we have \(ed = de\), as required.

$(2)\Rightarrow (1)$. Assume \(R\) is a strongly $\Delta$-clean ring, and let \(a \in R\) be arbitrary. Suppose \(a+1 = e + d\) is a strongly $\Delta$-clean representation. Then, \(a = e + (d-1)\), which is obviously a clean representation. Suppose now \(e + u = f + v\) are two clean representations. Then, by \cite [Proposition 2.3]{kkqt}, we infer \(e - f = v - u \in \Delta(R)\). Since all idempotents are central, we may write \(e - f = (e-f)^3\), and so \((e-f)^2 \in \Delta(R) \cap Id(R) = \{0\}\). Therefore, \(e - f = (e-f)^3 = (e-f)(e-f)^2 = 0\). Hence, \(e = f\), as needed.
\end{proof}

We say that an element $a$ in a ring $R$ is {\it uniquely $\Delta$-clean} if $a = e + d$, where $e^2 = e\in R$, $d \in \Delta(R)$ and this representation is unique. Thus, a ring $R$ is called {\it uniquely $\Delta$-clean} if every its element is uniquely $\Delta$-clean.

\begin{lemma}\label{uniquely}
Every idempotent in an uniquely $\Delta$-clean ring is central.
\end{lemma}

\begin{proof}
Let $R$ be an uniquely $\Delta$-clean ring. Suppose that $e^2=e\in R$. For any $r\in R$, we know that $e+(er-ere)$ is an idempotent and $(er-ere)$ is a nilpotent. On the other hand, $R$ is $\Delta U$ having in mind \cite[Theorem 4.2]{kkqt} and hence $Nil(R)\subseteq \Delta(R)$ taking into account \cite [Proposition 2.4(3)]{kkqt}. Therefore, $(er-ere)\in \Delta(R)$. Hence, $e+(er-ere)+0=e+(er-ere)$ are two $\Delta$-clean representations. It, thus, follows that $er=ere$ and, in a way of similarity, $re=ere$, as asked for.
\end{proof}

In addition to Theorem~\ref{cor 2}, we record the following necessary and sufficient condition.

\begin{theorem}\label{cor 3}
Let $R$ be a ring. Then, the following two items are equivalent:
	
(1) $R$ is an uniquely $\Delta$-clean ring;
	
(2) $R$ is a strongly $\Delta$-clean with all idempotents central.
\end{theorem}

\begin{proof}
$(1)\Rightarrow(2)$. Under presence of (1), we conclude that $R$ is abelian $\Delta$-clean, and hence it is strongly $\Delta$-clean.
	
$(2)\Rightarrow (1)$. Suppose $e+d=f+d'$ are two $\Delta$-clean representations. It suffices to establish that $e=f$. Since $\Delta(R)$ is a subring of $R$ and does not contain nonzero idempotents, we discover $e-f\in \Delta(R)$. Hence, $e-f=0$. So, $e=f$, as required.
\end{proof}

As a direct consequence, we yield:

\begin{corollary}
A ring $R$ is uniquely $\Delta$-clean if, and only if, $R$ is uniquely clean.
\end{corollary}

\begin{lemma} \label{ni idempotent}
Let $R$ be a strongly $\Delta$-clean ring. Then, $R$ is local if, and only if, it has no non-trivial idempotents.
\end{lemma}

\begin{proof}
If $R$ is local, then clearly $R$ has no non-trivial idempotents. Now, suppose that $R$ has no non-trivial idempotents. Then, for every $a \in R$, we have either $a \in \Delta(R)$ or $a-1 \in \Delta(R)$, which in turn assures that $R = \Delta(R) \cup U(R)$. Moreover, it can readily be proved that $R$ is local if, and only if, $R = \Delta(R) \cup U(R)$, as pursued.
\end{proof}

Our next valuable criterion is the following one.

\begin{theorem}\label{boolean}
A ring $R$ is strongly $\Delta$-clean with $J(R)=(0)$ if, and only if, $R$ is boolean.
\end{theorem}

\begin{proof}
Using Lemma \ref{lemma 4}, we conclude that $R$ is reduced, and hence it is abelian. Now, by virtue of
Theorem~\ref{cor 2}, we detect that $R$ is uniquely clean. Thus, \cite[Theorem 19]{nzu} is applicable to infer that  $R$ is boolean, as stated.
\end{proof}

As an automatic consequence, we yield:

\begin{corollary}\label{booli}
If $R$ is a strongly $\Delta$-clean ring, then the quotient-ring $R/J(R)$ is boolean.
\end{corollary}

A ring $R$ is said to be {\it right (resp., left) quasi-duo} if every maximal right (resp., left) ideal of $R$ is an ideal.

\medskip

We now derive the following.

\begin{corollary}
Every strongly $\Delta$-clean ring is right (left) quasi-duo.
\end{corollary}

\begin{proof}
Suppose $R$ is a strongly $\Delta$-clean ring. In view of Corollary \ref{booli}, $R/J(R)$ is boolean and hence it is commutative. Let $M$ be a right (left) maximal ideal of $R$. Then, one checks that $M/J(R)$ is an ideal of $R/J(R)$. Suppose that $x\in M$ and $r\in R$. Then, $\overline{rx} \in M/J(R)$, and thus $rx\in M+J(R)\subseteq M$. This shows that $M$ is an ideal of $R$. Therefore, $R$ is right (resp., left) quasi-duo, as asserted.
\end{proof}

We complete this section with the following suspected relationship between strongly nil-clean and strongly $\Delta$-clean rings. We recollect that, having in mind \cite{dl} or \cite{kwz}, a ring $R$ is strongly nil-clean exactly when $R/J(R)$ is boolean and $J(R)$ is nil.

\medskip

Specifically, the following is valid.

\begin{theorem}\label{nil}
A ring $R$ is strongly nil-clean if, and only if, $R$ is strongly $\Delta$-clean and $\Delta(R)$ is nil (that is, $\Delta(R)\subseteq Nil(R)$).
\end{theorem}

\begin{proof}
Assume $R$ is a strongly nil-clean ring. Then, working with \cite[Theorem 2.7]{kwz}, we deduce that $R/J(R)$ is Boolean and $J(R)$ is nil. First, we intend to show that $\Delta(R)$ is nil. Let $r \in \Delta(R)$, and suppose $r = e + q$ is a strongly nil-clean representation. Then, $$1 - e = (1 + q) - r \in U(R) + \Delta(R) \subseteq U(R),$$ which ensures that $1 - e = 1$, and thus $e = 0$. Therefore, $r = q \in \text{Nil}(R)$. On the other hand, knowing that $R/J(R)$ is Boolean, it follows that $ Nil(R) \subseteq J(R) \subseteq \Delta(R)$. Hence, it follows at once that $Nil(R) = \Delta(R)$, which gives that $R$ is a strongly $\Delta(R)$-clean ring.

Conversely, as $R$ is strongly $\Delta$-clean, we detect from Lemma~\ref{lemma 3} that $a-a^2\in \Delta(R)$ and, by hypothesis, $a-a^2\in Nil(R)$. Consequently, viewing \cite[Theorem 3]{hty}, $R$ is strongly nil-clean, as claimed.
\end{proof}

It is worthy of noticing that an other proof of the last result can be reached by using Corollary~\ref{booli} and the comments for characterizing strong nil-cleanness quoted before the mentioned theorem.

\section{Some Classes of Strongly $\Delta$-Clean Rings}

Let $A$, $B$ be two rings, and let $M$, $N$ be an $(A,B)$-bi-module and a $(B,A)$-bi-module, respectively. Also, we consider the two bi-linear maps $$\phi :M\otimes_{B}N\rightarrow A$$ and $$\psi:N\otimes_{A}M\rightarrow B$$ that apply to the following properties:
$$Id_{M}\otimes_{B}\psi =\phi \otimes_{A}Id_{M},Id_{N}\otimes_{A}\phi =\psi \otimes_{B}Id_{N}.$$
For $m\in M$ and $n\in N$, we define $mn:=\phi (m\otimes n)$ and $nm:=\psi (n\otimes m)$. Now, the $4$-tuple $R=\begin{pmatrix}
	A & M\\
	N & B
\end{pmatrix}$ becomes to an associative ring with obvious matrix operations that is called a {\it Morita context ring}. Denote two-sided ideals $Im \phi$ and $Im \psi$ to $MN$ and $NM$, respectively, that are called the {\it trace ideals} of the Morita context ring. A Morita context $\begin{pmatrix}
	A & M\\
	N & B
\end{pmatrix}$ is called {\it trivial} if the context products are trivial, i.e., $MN=0$ and $NM=0$. We now see that
$$\begin{pmatrix}
	A & M\\
	N & B
\end{pmatrix}\cong T(A\times B, M\oplus N),$$
where
$\begin{pmatrix}
	A & M\\
	N & B
\end{pmatrix}$ is a trivial Morita context, bearing in mind \cite{20}.

\medskip

The next preliminary technicality makes sense.

\begin{lemma}\label{rad morita}
Let $R = \begin{pmatrix}
A & M \\
N & B
\end{pmatrix}$ be a Morita context such that $MN \subseteq J(A)$ and $NM \subseteq J(B)$. Then, the following three statements hold:

(1) \quad  $J(R) = \begin{pmatrix}
    J(A) & M \\
    N & J(B)
    \end{pmatrix}$ and hence $R/J(R) \cong [A/J(A)] \times [B/J(B)]$.

(2)    $U(R) =\begin{pmatrix}
    U(A) & M \\
    N & U(B)
    \end{pmatrix}$.

(3)    $\Delta(R) =  \begin{pmatrix}
    \Delta(A) & M \\
    N & \Delta(B)
    \end{pmatrix}$.
\end{lemma}

\begin{proof}
Consulting with \cite[Lemma 3.1]{tang}, it suffices to prove only (3). To this end, suppose $D=\begin{pmatrix}
d & x \\
y & b
\end{pmatrix} \in \Delta(R)$. So, via point (2), for any $u,v \in U(A),U(B)$ we have $U=\begin{pmatrix}
u & 0 \\
0 & v
\end{pmatrix} \in U(R)$, and thus $1-UD \in U(R)$. This leads to $1-ud, 1-vb \in U(A), U(B)$ and, consequently, $d,b \in \Delta(A),\Delta(B)$.

Now, suppose $$D=\begin{pmatrix}
d & x \\
y & b
\end{pmatrix} \in \begin{pmatrix}
\Delta(A) & M \\
N & \Delta(B)
\end{pmatrix}$$ and $$U=\begin{pmatrix}
u & x' \\
y' & v
\end{pmatrix} \in U(R).$$ Then, since $MN \subseteq J(A)$ and $NM \subseteq J(B)$, we obtain
\[
1- UD=\begin{pmatrix}
1-ud-x'y & * \\
* & 1-vb-y'x
\end{pmatrix} \in \begin{pmatrix}
U(A) & M \\
N & U(B)
\end{pmatrix} = U(R),
\]
as expected
\end{proof}

We are now in a position to establish the following.

\begin{theorem}\label{com morita}
Let $A$ be a commutative subring of $B$ with $\Delta(A)1_B \subseteq \Delta(B)$, and let $R = \begin{pmatrix}
A & M\\
N & B
\end{pmatrix}$ be a trivial Morita context such that, for every $u \in U(A)$ and $m,n \in M,N$, the equalities $um=m(u1_B)$ and $nu=(u1_B)n$ are fulfilled. Then, the following two statements are equivalent:

(1) \quad $R$ is a strongly $\Delta$-clean ring.

(2) \quad $A$ and $B$ are strongly $\Delta$-clean rings.

\end{theorem}

\begin{proof}
Set $e:=\begin{pmatrix}
1 & 0\\
0 & 0
\end{pmatrix}$. Thus, $e \in R$ is an idempotent with $eRe \cong A$ and $(1-e)R(1-e) \cong B$. Therefore, Corollary \ref{corner ring} gives that $A$ and $B$ are strongly $\Delta$-clean rings.

Conversely, suppose both $A$ and $B$ are strongly $\Delta$-clean. Take $S=\begin{pmatrix}
a & m\\
n & b
\end{pmatrix} \in R$, where $a=e+d$ and $b=f+d'$ are strongly $\Delta$-clean representations in $A$ and $B$, respectively. Then, we receive
\[
\gamma =d1_B -d' +(1-2f)1_B \in \Delta(B)+\Delta(B)+U(B) \subseteq \Delta(B)+U(B) \subseteq U(B),
\]
so we may define
\[
\alpha := (em-mf)\gamma^{-1}, \quad \beta := \gamma^{-1} (ne-fn).
\]
Now, we manage to show that
\[
S=\begin{pmatrix}
a & m\\
n & b
\end{pmatrix} =\begin{pmatrix}
e & \alpha\\
\beta & f
\end{pmatrix} +\begin{pmatrix}
d & m-\alpha\\
n-\beta & d'
\end{pmatrix}
\]
is a strongly $\Delta$-clean representation for $S$. In fact, a consultation with Lemma \ref{rad morita} enables us that $D :=\begin{pmatrix}
d & m-\alpha\\
n-\beta & d'
\end{pmatrix} \in \Delta(R)$.

Next, we prove that $E :=\begin{pmatrix}
e & \alpha\\
\beta & f
\end{pmatrix}$ is an idempotent. Indeed, from the definitions of $\alpha$ and $\beta$ and the commutativity of $A$, we arrive at:
\begin{align*}
(i) & \quad e\alpha +\alpha f=(em-emf)\gamma^{-1}+(emf-mf)\gamma^{-1}=(em-mf)\gamma^{-1}=\alpha. \\
(ii) & \quad \beta e+f\beta=\gamma^{-1} (ne-fne)+\gamma^{-1} (fne-fn)=\gamma^{-1} (ne-fn)=\beta.
\end{align*}
Thus, from (i) and (ii), we get
\[
E^2=\begin{pmatrix}
e & e\alpha +\alpha f\\
\beta e+f\beta & f
\end{pmatrix}= \begin{pmatrix}
e & \alpha\\
\beta & f
\end{pmatrix} =E.
\]
So, $E$ is really an idempotent.

Furthermore, it remains to demonstrate that $ED=DE$. In this aspect, since $a=e+d$ and $b=f+d'$ are strongly $\Delta$-clean representations, we write $(ED)_{11}=(DE)_{11}$ and $(ED)_{22}=(DE)_{22}$. And since $\alpha = (em-mf)(d1_B -d' +(1-2f))^{-1}$, we can write
\[
em-me=\alpha d1_B - \alpha d' +\alpha(1-2f)=d \alpha - \alpha d' +\alpha - 2\alpha f = d \alpha - \alpha d' +e \alpha - \alpha f
\]
which, in turn, assures at once that
\[
e(m- \alpha ) +\alpha d'=d\alpha +(m- \alpha )f \Rightarrow (ED)_{12}=(DE)_{12}.
\]
Analogously, from the given above definition of $\beta$, we can illustrate that $(ED)_{21}=(DE)_{21}$. Therefore, combining these two equalities, $R$ is a strongly $\Delta$-clean ring, as promised.
\end{proof}

To visualize the last assertion, we consider the trivial Morita context $R_1 =
\begin{pmatrix}
\mathbb{Z}_4 & 2\mathbb{Z}_4 \\
2\mathbb{Z}_4 & \mathbb{Z}_4
\end{pmatrix}$ and
$R_2 =
\begin{pmatrix}
\mathbb{Z}_4 & 2\mathbb{Z}_4 \\
0 & \mathbb{Z}_4
\end{pmatrix}$,
where the context products are the same as the product in $\mathbb{Z}_4$. Thereby, $R_1$ and $R_2$ are both strongly $\Delta$-clean in virtue of Theorem~\ref{com morita}.

\medskip

As a non-trivial consequence, we obtain:

\begin{corollary} \label{triangular}
Let $R$ be a commutative ring. Then, the following four statements are equivalent:
	
(1) $R$ is strongly $\Delta$-clean.
	
(2) $R$ is uniquely clean.
	
(3) $T_n(R)$ is strongly $\Delta$-clean for some $n \geq 1$.
	
(4) $T_n(R)$ is strongly $\Delta$-clean for all $n \geq 1$.
\end{corollary}

A family of strongly $\Delta$-rings that are {\it not} uniquely clean are subsumed by Theorem~\ref{triangular}.

\medskip

Our next basic result is the following necessary and sufficient condition.

\begin{theorem}\label{loc mori}
Let $R = \begin{pmatrix}
A & M \\
N & B
\end{pmatrix}$ be a trivial Morita context, where $A$ and $B$ are local (or have no non-trivial idempotents) rings. Then, $R$ is a strongly $\Delta$-clean ring if, and only if, $A/J(A) \cong  \mathbb{F}_2 \cong B/J(B)$ and, for every $a,b \in U(A),J(B)$ and $m,n \in M,N$, there are $x,y \in M,N$ such that
\begin{align*}
m &= ax - xb, \\
n &= ya - by.
\end{align*}
\end{theorem}

\begin{proof}
Reviewing Lemma \ref{ni idempotent}, we know that in strongly $\Delta$-clean rings $R$ is local if, and only if, it has no non-trivial idempotents. Considering this fact, with no loss of generality, we will give a proof only for local rings.

Put $e :=\begin{pmatrix}
1 & 0 \\
0 & 0
\end{pmatrix}$. So, $e \in R$ is an idempotent with $eRe \cong A$ and $(1-e)R(1-e) \cong B$. In view of Corollary \ref{corner ring}, $A$ and $B$ are both strongly $\Delta$-clean rings. Moreover, it follows from Corollary \ref{booli} that both $A/J(A)$ and $B/J(B)$ are Boolean, and since $A$ and $B$ are local, we conclude that $A/J(A) \cong  \mathbb{F}_2  \cong B/J(B)$.

Via direct computations, it is easy to show that
\[
\text{Id}(R) = \left\{0,1, \begin{pmatrix}
1 & x \\
y & 0
\end{pmatrix},
\begin{pmatrix}
0 & x' \\
y' & 1
\end{pmatrix} : x,x' \in M \text{ and } y,y' \in N\right\}.
\]

Let $a,b \in U(A),J(B)$ and $m,n \in M,N$. Consider now $S :=\begin{pmatrix}
a & m \\
n & b
\end{pmatrix}$. Since, for any $s,k \in M,N$, neither $S-0$, $S-1$, nor $S-\begin{pmatrix}
0 & s \\
k & 1
\end{pmatrix}$ belongs to $\Delta(R)$, there must exist $x,y \in M,N$ such that $S-\begin{pmatrix}
1 & x \\
y & 0
\end{pmatrix} \in \Delta(R)$, and they obviously commute. This forces that
\[
m = ax - xb \quad \text{and} \quad n = ya - by.
\]

Now, assume $A/J(A) \cong  \mathbb{F}_2  \cong B/J(B)$ and that, for any $a,b \in U(A),J(B)$ and $m,n \in M,N$, there exist $x,y \in M,N$ satisfying $m=ax-xb$ and $n=ya-by$. We prove $R$ is strongly $\Delta$-clean. To that goal, set $S :=\begin{pmatrix}
a & m \\
n & b
\end{pmatrix}$, and consider four distinguish cases:

\begin{enumerate}
\item If $a \in U(A)$ and $b \in U(B)$, then from Theorem \ref{1}, $a-1\in \Delta(A)=J(A)$ and $b-1 \in \Delta(B)=J(B)$. Thus, by Lemma \ref{rad morita},
\[
S-1 = \begin{pmatrix}
a-1 & m \\
n & b-1
\end{pmatrix} \in J(R) \subseteq \Delta(R),
\]
so $S$ has a strongly $\Delta$-clean decomposition.

\item If $a \in J(A)$ and $b \in J(B)$, then from Lemma \ref{rad morita}, $S \in J(R)\subseteq \Delta(R)$, so $S$ has a strongly $\Delta$-clean decomposition.

\item If $a \in U(A)$ and $b \in J(B)$, since $A$ is a $\Delta U$ ring, we find $a-1 \in \Delta(A)$. By assumption, there are $x,y \in M,N$ such that $m=ax-xb$ and $n=ya-by$. Then,
\[
S := \begin{pmatrix}
a & m \\
n & b
\end{pmatrix} = \begin{pmatrix}
1 & x \\
y & 0
\end{pmatrix} + \begin{pmatrix}
a-1 & m-x \\
n-y & b
\end{pmatrix}
\]
gives a strongly $\Delta$-clean decomposition of $S$.

\item If $a \in J(A)$ and $b \in U(B)$, since $B$ is a $\Delta U$ ring, we have $b-1 \in \Delta(B)=J(B)$ and $a-1 \in U(A)$. By assumption, there are $x,y \in M,N$ such that
\[
m = (a-1)x - x(b-1) \quad \text{and} \quad n = y(a-1) - (b-1)y,
\]
which means $m=ax-xb$ and $n=ya-by$. Thus,
\[
S = \begin{pmatrix}
a & m \\
n & b
\end{pmatrix} = \begin{pmatrix}
0 & -x \\
-y & 1
\end{pmatrix} + \begin{pmatrix}
a & m+x \\
n+y & b-1
\end{pmatrix}
\]
gives a strongly $\Delta$-clean decomposition of $S$.
\end{enumerate}

From these four cases, we conclude that $R$ is a strongly $\Delta$-clean ring, as wanted.
\end{proof}

As three different consequences, we yield:

\begin{corollary}
Let $R = \begin{pmatrix}
A & M \\
0 & B
\end{pmatrix}$ be a formal triangular matrix, where $A$ and $B$ are local (or have no non-trivial idempotents) rings, and $M$ an $(A, B)$-bimodule. Then, $R$ is strongly $\Delta$-clean if, and only if,

(1) $A/J(A) \cong \mathbb{F}_2 \cong B/J(B)$.

(2) If $a \in U(A)$, $b \in J(B)$ and $m \in M$, there is $x \in V$ such that $m = ax - xb$.
\end{corollary}

Let $ a \in R $. The mappings $ l_a: R \to R $ and $ r_a: R \to R $ represent the abelian group endomorphisms defined by $ l_a(r) = ar $ and $ r_a(r) = ra $ for all $ r \in R $. Consequently, the expression $ l_a - r_b $ defines an abelian group endomorphism such that $ (l_a - r_b)(r) = ar - rb $ for any $ r \in R $. In accordance with
\cite{diesl}, a local ring $ R $ is classified as {\it bleached} if, for any $ a \in U(R) $ and $ b \in J(R) $, both
$l_a - r_b$ and $l_b - r_a$ are surjective.

The category of bleached local rings includes many well-known examples, such as commutative local rings, local rings with nil Jacobson radicals, and local rings for which some power of each element of their Jacobson radical is central (see cf. \cite[Example 13]{bdd}).

\begin{corollary}\label{11}
Let $R$ be a local ring and $n\geq 2$. Then, the following two conditions are equivalent:
	
(1) $T_n(R)$ is a strongly $\Delta$-clean ring.
	
(2) \( R \) is a bleached ring and \( R/J(R) \cong \mathbb{F}_2 \).
\end{corollary}

\begin{proof}
$(1) \Rightarrow (2)$. As $T_n(R)$ is a strongly $\Delta$-clean ring, employing Corollary \ref{corner ring} we also derive that $T_2(R)$ is a strongly $\Delta$-clean ring. Therefore, Theorem \ref{loc mori} allows us to conclude that $R/J(R) \cong \mathbb{F}_2$ and $R$ is a bleached ring, as needed.

$(2) \Rightarrow (1)$. Because \cite[Theorem 4.4]{csj} is true, $T_n(R)$ is strongly $J$-clean. However, we know that $J(R)\subseteq \Delta(R)$ holds always, whence $T_n(R)$ is a strongly $\Delta$-clean ring, as required.
\end{proof}

\begin{corollary}
Let $R$ be a commutative local ring. Then, the following two conditions are equivalent:
	
(1) $T_n(R)$ is a strongly $\Delta$-clean ring.
	
(2) \( R/J(R) \cong \mathbb{Z}_2 \).
\end{corollary}

\begin{proof}
This is pretty obvious given the fact that every commutative local ring is bleached.
\end{proof}

If we want to see a necessary and sufficient condition for a trivial Morita context to be strongly $\Delta$-clean, we can observe it in the following result. Since the proof is similar to Theorem \ref{loc mori}, we omit the proof.

\begin{theorem}\label{mor fin}
Let $R = \begin{pmatrix}
A & M \\
N & B
\end{pmatrix}$ be a trivial Morita context. Then, $R$ is a strongly $\Delta$-clean ring if, and only if, for every $a,b \in A,B$ and $m,n \in M,N$, there are $e,f \in \text{Id}(A),\text{Id}(B)$ and $x,y \in M,N$ such that
\begin{align*}
em - mf = ax - xb,\quad  ex + xf = x,\quad a - e \in \Delta(A) \:\:\text{ and } \:\: ae = ea, \\
ne - fn = ya - by,\quad  fy + ye = y, \: \:\quad b - f \in \Delta(B)\:\: \text{ and }\:\: bf = fb.
\end{align*}
\end{theorem}

With the above theorem in mind, we could have deduced Theorem~\ref{loc mori}, but since local rings have their own unique beauty, we refrained from doing so.

The following results can elementarily be derived from Theorem~\ref{mor fin}.

\begin{theorem}
Let $R = \begin{pmatrix}
A & M \\
N & B
\end{pmatrix}$ be a trivial Morita context with $A$ and $B$ abelian. Then, $R$ is a strongly $\Delta$-clean ring if, and only if, $A$ and $B$ are uniquely $\Delta$-clean and, for each $a,b \in A,B$ and $m,n \in M,N$, where $a = e + d$, $b = f + d'$ are the unique $\Delta$-clean decompositions in $A$ and $B$, respectively, there are $x,y \in M,N$ such that
\begin{align*}
&em - mf = ax - xb, \quad
ex + xf = x, \\
&ne - fn = ya - by, \quad\;\;
fy + ye = y.
\end{align*}
\end{theorem}

Let $R$ be a ring and $M$ a bi-module over $R$. The {\it trivial extension} of $R$ and $M$ is defined as
\[ T(R, M) = \{(r, m) : r \in R \text{ and } m \in M\}, \]
with addition defined componentwise and multiplication defined by
\[ (r, m)(s, n) = (rs, rn + ms). \]
One observes that the trivial extension $T(R, M)$ is isomorphic to the subring
\[ \left\{ \begin{pmatrix} r & m \\ 0 & r \end{pmatrix} : r \in R \text{ and } m \in M \right\} \]
of the formal $2 \times 2$ matrix ring $\begin{pmatrix} R & M \\ 0 & R \end{pmatrix}$, and likewise $T(R, R) \cong R[x]/\left\langle x^2 \right\rangle$. We also note that the set of units of the trivial extension $T(R, M)$ is exactly \[ U(T(R, M)) = T(U(R), M). \]
Moreover, knowing \cite {kkqt}, we write
\[ \Delta(T(R, M)) = T(\Delta(R), M). \]

\begin{corollary}
Suppose $R$ is a ring that has no non-trivial idempotent elements, and $M$ is a bi-modulo over $R$. Then, the following two conditions are equivalent:

(1) $T(R, M)$ is a strongly $\Delta$-clean ring.

(2) \( R/J(R) \cong \mathbb{Z}_2 \).
\end{corollary}

\begin{corollary}
If the trivial extension $T(R,M)$ is a strongly $\Delta$-clean ring, then $R$ is a strongly $\Delta$-clean ring. The converse holds, provided $em=me$ for all $m\in M$ and $e\in Id(R)$.
\end{corollary}

\begin{corollary}
If the trivial extension $T(R,R)$ is a strongly $\Delta$-clean ring, then $R$ is a strongly $\Delta$-clean ring. The converse holds when $R$ is abelian.
\end{corollary}

We are now prepared to prove the following.

\begin{lemma}\label{pur mori}
Let $R$ be the ring of a Morita context $\begin{pmatrix}
A & M\\
N & B
\end{pmatrix}$.
If $R$ is strongly $\Delta$-clean, then $A$ and $B$ are strongly $\Delta$-clean, and $MN,NM \subseteq J(A),J(B)$.
\end{lemma}

\begin{proof}
By analogy with Theorem \ref{com morita}, we can show that $A$ and $B$ are strongly $\Delta$-clean rings. It follows from Lemma~\ref{lemma 4} that $\text{Nil}(R) \subseteq J(R)$ and, therefore,
\[
\begin{pmatrix}
0 & M \\
0 & 0
\end{pmatrix},
\begin{pmatrix}
0 & 0 \\
N & 0
\end{pmatrix} \in J(R)
\]
which insures that
\[
\begin{pmatrix}
0 & M \\
N & 0
\end{pmatrix} \in J(R).
\]
Let $n \in N$ and $m \in M$ be arbitrary elements. Then, we have
\[
\begin{pmatrix}
mn & 0 \\
0 & nm
\end{pmatrix} =
\begin{pmatrix}
0 & m \\
n & 0
\end{pmatrix}
\begin{pmatrix}
0 & m \\
n & 0
\end{pmatrix} \in J(R),
\]
and, consequently,
\[
\begin{pmatrix}
MN & M \\
N & NM
\end{pmatrix} \in J(R).
\]
Thus, from \cite[Theorem 1]{rmor}, we conclude that $NM \subseteq J(A)$ and $MN \subseteq J(B)$, as desired.
\end{proof}

We are now attacking the following equivalencies.

\begin{theorem}\label{eq mori}
Let $R$ be the ring of a Morita context $\begin{pmatrix}
A & M\\
N & B
\end{pmatrix}$. If $J(A)$ and $J(B)$ are nilpotent ideals of $A$ and $B$, respectively, then
the following two assertions are equivalent:
\begin{enumerate}
    \item[(1)] $R$ is strongly $\Delta$-clean.
    \item[(2)] $A$ and $B$ are strongly $\Delta$-clean and $MN, NM \subseteq J(A), J(B)$.
\end{enumerate}
\end{theorem}

\begin{proof}
$(1) \Rightarrow (2)$. This is obvious following Lemma \ref{pur mori}.

$(1) \Rightarrow (2)$. Assume that $A$ and $B$ are strongly $\Delta$-clean rings. Since $J(A)$ and $J(B)$ are nil, a combination of Corollary \ref{booli} and \cite[Theorem 2.7]{kwz} ensures that $A$ and $B$ are strongly nil-clean rings. Moreover, since $MN \subseteq J(A)$ and $NM \subseteq J(B)$, and $J(A)$ and $J(B)$ are nilpotent, it follows that $MN$ and $NM$ are also nilpotent. Therefore, \cite[Theorem 3.4]{kwz} assures that $R$ is a strongly nil-clean ring. Consequently, Theorem \ref{nil} forces that $R$ is a strongly $\Delta$-clean ring, as expected.
\end{proof}

As a concrete construction, considering the trivial Morita context $R_1 =
\begin{pmatrix}
\mathbb{Z}_4 & 2\mathbb{Z}_4 \\
\mathbb{Z}_4 & \mathbb{Z}_4
\end{pmatrix}$,
$R_2 =
\begin{pmatrix}
\mathbb{Z}_4 & \mathbb{Z}_4 \\
2 \mathbb{Z}_4 & \mathbb{Z}_4
\end{pmatrix}$,
where the context products are the same as the product in $\mathbb{Z}_4$, we infer with the help of
Theorem~\ref{eq mori} that $R_1$ and $R_2$ are both strongly $\Delta$-clean rings.

\section{Group Rings}

As usual, for an arbitrary ring $R$ and an arbitrary group $G$, the double-letter $RG$ stands for the {\it group ring} of $G$ over $R$. Standardly, $\varepsilon(RG)$ denotes the kernel of the classical {\it augmentation map} $\varepsilon: RG\to R$ defined by $\varepsilon (\displaystyle\sum_{g\in G}a_{g}g)=\displaystyle\sum_{g\in G}a_{g}$, and this ideal is called the {\it augmentation ideal} of $RG$.

Besides, a group $G$ is called a {\it $p$-group} if every element of $G$ has finite order which is a power of the prime number $p$. Moreover, a group $G$ is said to be {\it locally finite} if every finitely generated subgroup is finite.

The following result is crucial for our presentation here.

\begin{lemma}\label{prop group ring}
Let $R$ be a strongly $\Delta$-clean ring and let $G$ be a locally finite 2-group. Then, the following three conditions hold:

(1) $\varepsilon(RG) \subseteq J(RG)$.

(2) $RG/J(RG)$ is Boolean.

(3) $J(RG) = \{x \in RG : \varepsilon(x) \in J(R)\}$.
\end{lemma}

\begin{proof}
(1) Applying Lemma \ref{lemma 5}, we discover $2 \in J(R)$. Thus, from \cite[Lemma 2]{zucc}, we obtain $\varepsilon(RG) \subseteq J(RG)$.

\medskip

(2) For any $\sum_{g \in G} a_gg \in RG$, we have
\[
\sum_{g \in G} a_gg = \sum_{g \in G} -a_g(1-g) + \sum_{g \in G} a_g \in \varepsilon(RG) + R.
\]
Therefore, from point (1), we get $RG = J(RG) + R$. Moreover, an application of \cite[Proposition 9]{conel} guarantees that $J(RG) \cap R = J(R)$, whence
\[
R/J(R) \cong R/(J(RG)\cap R) \cong R+J(RG)/J(RG) = RG/J(RG).
\]
However, Corollary~\ref{booli} teaches that $R/J(R)$ is Boolean, and thus $RG/J(RG)$ is also Boolean.

\medskip

(3) Set $A := \{x \in RG : \varepsilon(x) \in J(R)\}$. Since $\varepsilon$ is surjective, we detect that $\varepsilon(J(RG)) \subseteq J(R)$, which means $J(RG) \subseteq A$. Now, let $x = \sum_{g \in G} a_gg \in A$, so $\varepsilon(x) \in J(R)$. Moreover, by (2), $RG/J(RG)$ is Boolean, and hence, for any $g \in G$, we deduce $1-g \in J(RG)$. Therefore,
\[
x = \sum_{g \in G} a_g(g-1) + \sum_{g \in G} a_g = \sum_{g \in G} a_g(g-1) + \varepsilon(x) \in J(RG),
\]
showing that $A \subseteq J(RG)$. Notice that, since $G$ is a locally finite group, \cite[Proposition 9]{conel} is applicable to derive that $J(RG) \cap R = J(R)$, as required.
\end{proof}

We now come to our main result in this section.

\begin{theorem}\label{1.6}
If the group ring $RG$ is strongly $\Delta$-clean, then $R$ is a strongly $\Delta$-clean ring and $G$ is a 2-group.
\end{theorem}

\begin{proof}
An exploitation of Lemma \ref{prop group ring} reaches us that $\varepsilon(RG) \subseteq J(RG)$. Therefore, combining Lemma \ref{lemma 0} and the fact that $RG/\varepsilon(RG) \cong R$, we conclude that $R$ is a strongly $\Delta$-clean ring.

On the other hand, thanking to \cite[Proposition 15]{conel}, we infer that $G$ is a $p$-group with $p \in J(R)$. However, Lemma \ref{lemma 0} enables us that $2 \in J(R)$, which yields $p=2$, because two different primes cannot simultaneously lie in $J(R)$.
\end{proof}

We now continue with the next technical claim.

\begin{lemma}
Let $R$ be a semi-abelian ring and let $G$ be a locally finite group. Then, the following two conditions are equivalent:

(1) $RG$ is a strongly $\Delta$-clean ring.

(2) $R$ is a strongly $\Delta$-clean ring and $G$ is a 2-group.
\end{lemma}

\begin{proof}
$(1) \Rightarrow (2)$. This follows at once from Theorem \ref{1.6}.

$(2) \Rightarrow (1)$. Assume $R$ is a strongly $\Delta$-clean ring. Then, for every prime ideal $P$, we find that $R/P \cong (R/J(R))/(P/J(R))$. But, since $R/J(R)$ is Boolean, $R/P$ is also Boolean. Moreover, exploiting \cite[Theorem 3.3]{wchen}, the quotient $R/P$ is local. Thus, as $R/P$ is Boolean, it must be that $J(R/P)=(\overline{0})$, which gives $R/P \cong \mathbb{F}_2$. Therefore, invoking \cite[Theorem 3.3]{chin}, $RG$ is a strongly $\pi$-regular ring.

Furthermore, since $RG/J(RG) \cong R/J(R)$ is Boolean, it follows that $RG$ is a UJ-ring. So, Theorem \ref{1} is a guarantor that $RG$ is a strongly $\Delta$-clean ring, as stated. Note that every strongly $\pi$-regular ring is strongly clean always.
\end{proof}

From the above lemma, we can conclude that, if $R$ is abelian and strongly $\Delta$-clean, and $G$ is a locally finite 2-group, then $RG$ is also a strongly $\Delta$-clean ring.

Nevertheless, we now attempt to prove this point differently in the following.

\begin{lemma}
Let $R$ be an abelian ring and let $G$ be a locally finite group. Then, the following two conditions are equivalent:

(1) $RG$ is a strongly $\Delta$-clean ring.

(2) $R$ is a strongly $\Delta$-clean ring and $G$ is a 2-group.
\end{lemma}

\begin{proof}
$(1) \Rightarrow (2)$. This follows immediately from Theorem \ref{1.6}.

$(2) \Rightarrow (1)$. Arguing as in point (2) of Lemma \ref{prop group ring}, we can show that $RG=J(RG)+R$. Choose $f \in RG$. We may, with no harm of generality, assume that $f=a+j \in R +J(RG)$. Let $a=e+d$ be a strongly $\Delta$-clean decomposition. Consulting with Lemma \ref{prop group ring}, one has that $d-d^2=d(1-d)\in J(RG)$. Since $d \in \Delta(R)$, we know that $1-d \in U(R)$, and thus $d \in J(RG)$. Therefore, we can write $$f=e+(d +j) \in \text{Id}(RG)+J(RG).$$ Moreover, since $R$ is abelian, we extract $e(d+j)=(d+j)e$. Consequently, $RG$ is a strongly $\Delta$-clean ring, as formulated.
\end{proof}

%\begin{proof}
%If $DG$ is strongly $\Delta$-clean, then $D$ is strongly $\Delta$-clean and $G$ is a 2-group. But $D$ is boolean and local, so we must have $D\cong \mathbb Z_2$. The converse follows from \cite [Corollary 9]{cnzg}.
%\end{proof}

The next construction appears to be interesting.

\begin{example}
If $R$ is a commutative uniquely clean ring and $G$ is an abelian $2$-group, then $T_n(R)G\cong T_n(RG)$ is strongly $\Delta$-clean for all $n\geq 1$.
\end{example}

\begin{proof}
In fact, \cite[Theorem 12]{cnzg} claims that $RG$ is a commutative uniquely clean ring. Thus, it follows from
Theorem~\ref{triangular} that $T_n(R)G\cong T_n(RG)$ is strongly $\Delta$-clean for all $n\geq 1$. But, however, they are surely {\it not} uniquely clean unless $n=1$.
\end{proof}

Our final assertion is the following.

\begin{proposition}\label{group}
Let $R$ be a ring and let $G\neq 1$ be a locally finite group in which every finite subgroup has odd order. Then, $RG$ is not strongly $\Delta$-clean.
\end{proposition}

\begin{proof}
Suppose the contrary that $RG$ is strongly $\Delta$-clean. In virtue of Theorem~\ref{1.6}, we can conclude that $R$ is strongly $\Delta$-clean, and hence $\bar R=R/J(R)$ is boolean in conjunction with Corollary~\ref{booli}. But, since every finite subgroup has odd order, each natural $n$ has the property that $n\in |G|$ is a unit in $R$ (here $|G|$ is the set of orders of all finite subgroups in $G$). Moreover, as $\bar R$ is regular and $G$ is locally finite, $\bar RG$ is too regular via an exploitation of \cite[Theorem 3]{conel}. However, $\bar RG$ is strongly $\Delta$-clean, whence it is boolean, but this is an obvious contradiction because $G\neq 1$. Consequently, $RG$ is not strongly $\Delta$-clean, as asserted.
\end{proof}

In closing, in regard to Theorem~\ref{1}, we pose the following.

\medskip

\noindent{\bf Problem 1.} Decide when a strongly $\Delta$-clean ring is strongly $J$-clean.

\medskip

\noindent{\bf Problem 2.} Find a necessary and sufficient condition for a group ring to be strongly $\Delta$-clean.

\medskip
\medskip
\medskip
\medskip

%\noindent{\bf Acknowledgement.} The authors express their sincere gratitude to the expert referee for the numerous competent suggestions made which led to a substantial improvement of the exposition.

\noindent{\bf Funding:} This work is based upon research funded by Iran National Science Foundation (INSF) under project no. 4041928.

\end{document}